\documentclass[preprint,12pt,3p]{elsarticle}

\usepackage{amssymb}
\usepackage{amsthm}
\usepackage{amsmath}
\usepackage{mathrsfs}
\usepackage{color}
\usepackage[T1]{fontenc}
\usepackage{enumitem}
\usepackage{hyperref}
\hypersetup{
    colorlinks=true,
    linkcolor=red,
    filecolor=magenta,      
    urlcolor=cyan,
}
\newtheorem{theorem}{Theorem}

\newtheorem{remark}[theorem]{Remark}

\newtheorem{conjecture}[theorem]{Conjecture}

\newtheorem{proposition}[theorem]{Proposition}

\newtheorem*{conjecture*}{Conjecture}

\renewcommand{\epsilon}{\varepsilon}
\renewcommand{\phi}{\varphi}
\renewcommand{\kappa}{\varkappa}
\renewcommand{\theta}{\vartheta}

\begin{document}

\begin{frontmatter}

\title{A counterexample to a conjecture on the chromatic number of $r$-stable Kneser hypergraphs}

\author[label1]{Hamid Reza Daneshpajouh}
\address[label1]{
University of Nottingham Ningbo China, 199 Taikang E Rd, Yinzhou, Ningbo, Zhejiang, China, 315104}

\ead{Hamid-Reza.Daneshpajouh@nottingham.edu.cn}

\begin{abstract}
The main purpose of this note is to give a counterexample to the following conjecture, raised by Florian Frick [\textit{Int. Math. Res. Not. IMRN 2020 (13), 4037-4061 (2020)}].
\begin{conjecture*}
 Let $r\geq 3$ and let $\mathcal{F}$ be a set system. Then
 $$\chi\left(\textrm{KG}^r\left(\mathcal{F}_{r-stab}\right)\right)\geq\left\lceil\frac{cd_r\left(\mathcal{F}\right)}{r-1}\right\rceil.$$
\end{conjecture*}

\end{abstract}

\end{frontmatter}

\section{Introduction}
Throughout this note, the symbol $[n]$ is used for the set $\{1,\ldots, n\}$, the set of all $k$-subsets of $[n]$ is denoted by $\binom{[n]}{k}$, and the set of all subsets of a set $X$ is denoted by $2^X$. A hypergraph $\mathcal{H}=\left(V, E\right)$ is a pair $\left(V, E\right)$ where $V$ is a finite set of elements called vertices, and $E$ is a set of non-empty subsets of $V$ called edges. An $m$-coloring of a hypergraph $\mathcal{H}$ is a map $c: V\left(\mathcal{H}\right)\to\{1,\ldots, m\}$ with no monochromatic edge, i.e., $|c\left(e\right)|\geq 2$ for all $e\in\mathcal{H}$. We say a hypergraph is $m$-colorable if it admits an $m$-coloring. The chromatic number of $\chi\left(\mathcal{H}\right)$ of a hypergraph $\mathcal{H}$ is the minimum integer $m$ such that $\mathcal{H}$ is $m$-colorable. For an integer $r\geq 2$, $r$-colorability defect $cd_{r}\left(\mathcal{H}\right)$ of a hypergraph $\mathcal{H}=\left(V, E\right)$ is the minimum number of vertices that they must be removed from the vertex set of $\mathcal{H}$ such that the induced hypergraph on the remaining vertices is $r$-colorable. For a system $\mathcal{F}$ of subsets of a set and non-negative integer $r$, $\textrm{KG}^{r}\left(\mathcal{F}\right)$ is an $r$-uniform hypergraph whose vertices are all elements of $\mathcal{F}$ and edges are all sets $\{A_1, \ldots , A_r\}$ of $r$ vertices where $A_i\cap A_j$ for $i\neq j$. A subset $\sigma \subseteq [n]$ is called \textit{$s$-stable} if $s\leq |i-j|\leq n-s$ for $i, j\in \sigma$ distinct; Similarly $\sigma\subseteq [n]$ is called \textit{almost $s$-stable} if $s\leq |i-j|$ for $i, j\in \sigma$ distinct. For a system $\mathcal{F}$ of the subsets of $[n]$, the set of all $s$-stable (almost $s$-stable) members of $\mathcal{F}$ is denoted by $\mathcal{F}_{s-\text{stab}}$ ($\mathcal{F}_{\widetilde{s-\text{stab}}}$). Note that a set system $\mathcal{F}$ of subsets of $[n]$ can be seen as the hypergraph $\left([n], \mathcal{F}\right)$, and so these two words are used interchangeably through this note.

In 1978, Lov\'{a}sz~\cite{Lo78} proved that
$$\chi\left(\textrm{KG}^2\left(\binom{[n]}{k}\right)\right)=n-2\left(k-1\right),\quad \text{for all}\,\, n\geq 2k$$ 
and shortly afterward, Schrijver~\cite{Sch78} showed that even after deleting all non 2-stable vertices from $\binom{[n]}{k}$ the chromatic number is not changed, i.e., 
$$\chi\left(\textrm{KG}^2\left({{\binom{[n]}{k}}_{2-\text{stab}}}\right)\right)=n-2\left(k-1\right)\quad \text{for all}\,\, n\geq 2k.$$
Later, Alon--Frankl--Lov\'{a}sz~\cite{alon1986chromatic} generalized the Lov\'{a}sz's result as follows

$$
\chi \left({\textrm{KG}}^{r}\left(\binom{[n]}{k}\right)\right)=\left\lceil\frac{n-r(k-1)}{r-1}\right\rceil\quad\text{for all}\,\, n\geq rk\,\,\&\,\, r\geq 2.
$$
Ziegler~\cite{ziegler2002generalized} conjectured that the Alon--Frankl--Lov\'{a}sz result holds if even all the non $r$-stable vertices are removed, i.e.,
\begin{conjecture}\label{Conj: zig}
If $r\geq 2$ and $n\geq rk$, then
$$\chi \left(\textrm{KG}^r\left(\binom{[n]}{k}_{r-\text{stab}}\right)\right)=\left\lceil\frac{n-r(k-1)}{r-1}\right\rceil\quad\text{for all}\,\, n\geq rk.$$
\end{conjecture}

And finally, the following generalization of the Ziegler conjecture is raised in ~\cite{frick2020chromatic}.
\begin{conjecture}
\label{Conj: Florian}
 Let $r\geq 3$ and let $\mathcal{F}$ be a set system on the ground set $[n]$. Then
 $$\chi\left(\textrm{KG}^r\left(\left(\mathcal{F}_{r-stab}\right)\right)\right)\geq\left\lceil\frac{cd_r(\mathcal{F}}{r-1}\right\rceil.$$
\end{conjecture}

The main aim of this paper to disprove Conjecture~\ref{Conj: Florian}.

\section{Main Result}
Let $\mathcal{F}(n,r)=\binom{[n]}{2}\setminus\binom{[n]}{2}_{r-stab}$ be the set system on the ground set $[n]$.

\begin{proposition}
Let $r\geq 2$, and $n=kr+1$ for some $k\geq 1$. Then, we have 
$$cd_r\left(\mathcal{F}(n,r)\right)= 1.$$
\end{proposition}
\begin{proof}
First of all it is easy to see that $cd_r\left(\mathcal{F}(n,r)\right)\leq 1$. Indeed, after removing the vertex $n$, the induced family on the remaining vertices is $r$-colorable; for each $1\leq i\leq r$ assign the color $i$ to the vertices $i, i+r, i+2r, \ldots, i+(k-1)r$. Now, assume the contrary that is $cd_r\left(\mathcal{F}(n,r)\right)=0$. This means that there is a proper coloring of $\mathcal{F}(n,r)$ with $r$ colors, namely $1, \ldots, r$. First note that each pair of distinct vertices among $\{1, \ldots, r\}$ must receive different colors as any two of them present an element of $\mathcal{F}(n,r)$. Without loss of generality, assume that each vertex $i$ receives the color $i$ for $1\leq i\leq r$. Since $\{i, r+1\}\in\mathcal{F}(n,r)$ for every $2\leq i\leq r$, the vertex $r+1$ must be colored by $1$. This implies the vertex $r+2$ must be colored by $2$ as $\{r+1, r+2\}\in\mathcal{F}(n,r)$ and $\{i, r+2\}\in\mathcal{F}(n,r)$ for every $3\leq i\leq r$. Inductively, a similar argument shows that each of the vertex of the form $lr+j$ must receive the color $j$ for each $1\leq j\leq r$ and $0\leq l\leq k-1$. This implies that the vertex $n=kr+1$ must be colored with $1$. But, this is a contradiction as $\{1, n\}\in\mathcal{F}(n,r)$. Therefore,  $cd_r\left(\mathcal{F}(n,r)\right)= 1$ and this finishes the proof. 
\end{proof}
But on the other hand, by the definition, we have $\mathcal{F}(n,r)_{r-stab}=\emptyset$ for every $n$ which this implies that $\chi\left({\textrm{KG}}^r\left({\mathcal{F}(n,r)}_{r-stab}\right)\right)= 0$. Thus, the family $\mathcal{F}(n,r)$ when $n=kr+1$\footnote{Actually, a similar argument shows that the family $\mathcal{F}(n,r)$ provides a counterexample for Conjecture \ref{Conj: Florian} provided $n$ is not a multiply of $r$ and $n\geq r$.} for some $k\geq 1$ gives a counterexample to the mentioned conjecture.
\vspace{0.2cm}

At first, we thought that if we add the assumption that the $r$-stable part of the family $\mathcal{F}$ is not-empty in Conjecture~\ref{Conj: Florian}, then this version of the conjecture might be true. After trying to find some positive results in this direction for weeks, unfortunately, finally we came up with the following counterexample again. 

\begin{proposition}
Let $r\geq 2$, $n=r(2r-1)$ and set
$$\mathcal{F}=\mathcal{F}(n,r)\bigcup\{\{1+ir, 1+(i+1)r\} : i=0, \ldots 2r-3\}\bigcup\{\{(2r-2)r+1, 1\}\}.$$
We have
$$cd_r(\mathcal{F})\geq r\quad \&\quad \chi(\textrm{KG}^{r}(\mathcal{F}_{r-stab})))=1.$$
\end{proposition}
\begin{proof}
First note that
$$\mathcal{F}_{r-stab}=\{\{1+ir, 1+(i+1)r\} : i=0, \ldots 2r-3\}\bigcup\{\{(2r-2)r+1, 1\}\},$$
which implies $\chi\left(\textrm{KG}^{r}\left(\mathcal{F}_{r-stab}\right)\right)=1$ as
$\mathcal{F}_{r-stab}\neq\emptyset$ and
there are no $r$-pairwise disjoint sets in $\mathcal{F}_{r-stab}$. Indeed, the later property is simply because the size of the set $\bigcup \mathcal{F}_{r-stab}=\{1, 1+r, \ldots, 1+r(2r-2)\}$ is $2r-1$, and hence any collection of $2$-subsets of $\mathcal{F}_{r-stab}$ of size $r$ contains at least two sets with a non-empty intersection.  

For the other part, suppose the contrary, that is $cd_r\left(\mathcal{F}\right)\leq r-1$. So, we can remove a set $B$ of size $r-1$ of elements of the ground set $[n]$ such that the induced family on the rest of elements is $r$-colorable. First note that, $B$ must contain at least one element from each of the following sets
$$A_i=\{1+ir,\ldots , 1+\left(i+1\right)r\}\quad \&\quad \text{for}\quad i=0, \ldots, 2r-3.$$
This is because, if we do not delete an element from an $A_i$ for some $i$, then each elements of this set must receive a different colors as any $2$-subset of $A_i$ is in $\mathcal{F}$. But, then this is impossible as the size of $A_i$ is $r+1$ and we have just $r$ colors. Next, note that for $i < j$
$$A_i\cap A_j\neq\emptyset\quad\text{iff}\quad j= i+1,$$
and moreover $A_i\cap A_{i+1}=\{1+\left(i+1\right)r\}$ for $i=0, \ldots, 2r-3$.
These properties beside the fact that we have $2r-2$ sets $A_i$ and the size of $|B|=r-1$ uniquely determine $B$, i.e.,
$$B=\{1+r, 1+3r,\ldots, 1+\left(2r-3\right)r\}.$$

So, in particular, $B$ does not have any elements of the set  
$$C=\{1+\left(2r-2\right)r, 2+\left(2r-2\right)r, \ldots, \left(2r-1\right)r, 1 \}.$$ 
But, then again we need $r+1$ colors to color this part as any $2$-subset of $C$ appears in $\mathcal{F}$ and $|C|=r+1$. This contradiction finishes the proof. 
\end{proof}

We believe the following weaker version of Florian's conjecture might be true.
\begin{conjecture*}\label{Conj}
 Let $r\geq 2$ and let $\mathcal{F}$ be a set system. Then
 $$\chi\left(\textrm{KG}^r\left(\mathcal{F}_{\widetilde{r-\text{stab}}}\right)\right)\geq\left\lceil\frac{cd_r\left(\mathcal{F}\right))}{r-1}\right\rceil.$$
\end{conjecture*}

\begin{remark}
\normalfont
If Conjecture \ref{Conj} is true, then 
$$\left\lceil\frac{cd_r\left(\mathcal{F}\right)}{r-1}\right\rceil-\chi\left(\textrm{KG}^r\left(\mathcal{F}_{r-stab}\right)\right)\leq 1.$$
To confirm this claim it is enough to show that
$$\chi\left(\textrm{KG}^r\left(\mathcal{F}_{\widetilde{r-stab}}\right)\right)-\chi\left(\textrm{KG}^r\left(\mathcal{F}_{{r-stab}}\right)\right)\leq 1,$$
where $\mathcal{F}$ is an arbitrary family on the ground set $[n]$.
Let $c$ be an proper coloring of $\textrm{KG}^r\left(\mathcal{F}_{{r-stab}}\right)$. We extend this coloring to a proper coloring $c^{\prime}$ for $\textrm{KG}^r\left(\mathcal{F}_{\widetilde{r-stab}}\right)$ with one more new color, namely $\ast$. Define
\[ c^{\prime}\left(F\right)=\begin{cases} 
      c\left(F\right)\quad\, \text{if}\quad F\in \mathcal{F}_{{r-stab}}\\
      \ast\quad\quad\quad \text{if}\quad F\in \mathcal{F}_{\widetilde{r-stab}}\setminus \mathcal{F}_{{r-stab}}
   \end{cases}
\]
It is pretty easy to check that this map gives a proper coloring of $\textrm{KG}^r\left(\mathcal{F}_{\widetilde{r-stab}}\right)$. Indeed, note that if $$F\in \mathcal{F}_{\widetilde{r-stab}}\setminus \mathcal{F}_{{r-stab}}$$
then $F\cap\{1, \ldots, r-1\}\neq\emptyset$. Thus, there are two sets with a non empty intersection among any collection of $r$ such sets, which this verifies the claim.
\end{remark}


\medskip
 
\bibliographystyle{alpha}
\bibliography{main}

\end{document}